\documentclass[11pt]{amsart}
\usepackage[all]{xy}
\usepackage{a4wide}
\usepackage{amssymb}
\usepackage{amsthm}
\usepackage{amsmath}
\usepackage{amscd,enumitem}
\usepackage{verbatim}
\usepackage{float}
\usepackage{color}
\usepackage[all]{xy}
\usepackage{hyperref}
\textheight8.7in \textwidth6.6in \numberwithin{equation}{section}

\theoremstyle{plain}
\newtheorem{theorem}{Theorem}
\numberwithin{theorem}{section}
\newtheorem{corollary}[theorem]{Corollary}

\newtheorem{lemma}[theorem]{Lemma}
\newtheorem{proposition}[theorem]{Proposition}

\theoremstyle{definition}

\newtheorem{example}{Example}

\theoremstyle{remark}
\newtheorem*{remark}{Remark}

\newenvironment{psmallmatrix}{\left(\begin{smallmatrix}}{\end{smallmatrix}\right)}

\newcommand{\Stab}{\mathrm{Stab}}
\newcommand{\R}{\mathbb{R}}

\newcommand{\Z}{\mathbb{Z}}
\newcommand{\N}{\mathbb{N}}
\newcommand{\C}{\mathbb{C}}

\renewcommand{\H}{\mathbb{H}}

\newcommand{\ord}{{\text {\rm ord}}}

\newcommand{\dv}{\operatorname{div}}

\newcommand{\re}{\operatorname{Re}}
\newcommand{\im}{\operatorname{Im}}

\newcommand{\SL}{{\text {\rm SL}}}

\begin{document}

\title[On Divisors of Modular Forms]{On Divisors of Modular Forms}

\author{Kathrin Bringmann, Ben Kane, Steffen L\"obrich, Ken Ono, and Larry Rolen}

\dedicatory{In celebration of Don Zagier's 65th birthday.}

\address{Mathematical Institute, University of Cologne, Weyertal 86-90, 50931 Cologne, Germany}
\email{kbringma@math.uni-koeln.de}
%\email{steffen.loebrich@uni-koeln.de}

\address{Department of Mathematics, University of Hong Kong, Pokfluam, Hong Kong}
\email{bkane@maths.hku.hk}

\address{Mathematical Institute, University of Cologne, Weyertal 86-90, 50931 Cologne, Germany}
\email{steffen.loebrich@uni-koeln.de}

\address{Department of Mathematics and Computer Science, Emory University,
Atlanta, Georgia 30022} \email{ono@mathcs.emory.edu}

\address{Hamilton Mathematics Institute \& School of Mathematics, Trinity College, Dublin 2, Ireland}
\email{lrolen@maths.tcd.ie}

\thanks{The first and third author are supported by the Deutsche Forschungsgemeinschaft (DFG) Grant No. BR 4082/3-1. The second author was supported by grant project numbers 27300314, 17302515, and 17316416 of the Research Grants Council. The fourth author thanks the support of the NSF and the Asa Griggs Candler Fund.
}

\begin{abstract} The {\it denominator formula} for the Monster Lie algebra is the product expansion for the modular function $j(z)-j(\tau)$ given in terms of the Hecke system of $\SL_2(\Z)$-modular functions $j_n(\tau)$. It is prominent in Zagier's seminal paper  on traces of singular moduli, and in the Duncan-Frenkel  work on Moonshine.  The formula
is equivalent to the description of the generating function for the $j_n(z)$ as a weight 2 modular form with a  pole at  $z$.
Although these results rely on the fact that $X_0(1)$ has
genus 0, here we obtain a generalization, framed in terms of polar harmonic Maass forms, for all of the $X_0(N)$ modular curves. We use these functions to study divisors of modular forms.
\end{abstract}

\maketitle

\section{Introduction and statement of results}
As usual, let $j(\tau)$ be the $\SL_2(\Z)$-modular function defined by
\begin{align*}
	j(\tau)=\sum_{n=-1}^{\infty}c(n)e^{2\pi i n \tau}
	&:=\frac{E_4(\tau)^3}{\Delta(\tau)}=e^{-2\pi i \tau}+744+196884e^{2\pi i \tau}+\cdots,
\end{align*}
where $E_k(\tau)
	:=1-\frac{2k}{B_k}\sum_{n=1}^{\infty}\sigma_{k-1}(n)e^{2\pi i n \tau}$
is the weight $k\in 2\N$ Eisenstein series, $\sigma_\ell(n):=\sum_{d|\ell} d^\ell$, $B_k$ is the $k$th Bernoulli number, and $\Delta(\tau):=(E_4(\tau)^3-E_6(\tau)^2)/1728$.
By Moonshine  (for example, see \cite{Moonshine}), $j(\tau)$ is the McKay-Thompson series for the identity (i.e., its coefficients are the graded dimensions of the Monster module $V^{\natural}$).
Moonshine also offers the striking infinite product
$$
j(z)-j(\tau)
=e^{-2\pi iz} \prod_{m>0, ~n\in \Z}
\left(1-e^{2\pi i m z}e^{2\pi i n \tau}\right)^{c(mn)},
$$
the {\it denominator formula} for the Monster Lie algebra.
This  formula is equivalent to the following identity of Asai, Kaneko, and Ninomiya (see Theorem 3 of \cite{AKN})
%\footnote{We note that $H_z(\tau)$ is denoted $H_{\tau}(z)$ in \cite{BKO} where the roles of $\tau$ and $z$ are reversed.}
\begin{equation}\label{ZagierHFunction}
H_{z}(\tau):= \sum_{n=0}^{\infty} j_n(z)e^{2\pi i n \tau}=\frac{E_4(\tau)^2E_6(\tau)}{\Delta(\tau)} \frac{1}{j(\tau)-j(z)}=-\frac{1}{2\pi i} \frac{j'(\tau)}{j(\tau)-j(z)}.
\end{equation}
The functions $j_n(\tau)$ form a Hecke system. Namely, if we let
$j_0(\tau):=1$ and $j_1(\tau):=j(\tau)-744$, then the others are obtained
by applying the normalized Hecke operator $T(n)$
\begin{equation}\label{jnz}
j_n(\tau):=j_1(\tau) \ | \ T(n).
\end{equation}
%We note that each $j_n(\tau)$ is uniquely determined by the Fourier expansion
%$j_n(\tau)=e^{-2\pi i n \tau}+O(e^{2\pi i \tau}).$
%The function $j_n(\tau)$ is a monic integral degree $n$ polynomial in $j(\tau)$, a so-called {\it Faber polynomial}.
\begin{remark}
The functions $H_{z}(\tau)$ and $j_n(\tau)$ played central roles in Zagier's \cite{Zagier} seminal paper on traces of singular moduli and the Duncan-Frenkel work \cite{DF} on the Moonshine Tower.
Carnahan \cite{Carnahan} has obtained similar denominator formulas for completely replicable modular functions.
\end{remark}

If $z\in \H$, then $H_{z}(\tau)$ is a weight $2$ meromorphic modular form on $\SL_2(\Z)$ with a single pole (modulo $\SL_2(\Z)$) at the point $z$.
Using these functions, modular form avatars for divisors of meromorphic modular forms were defined in \cite{BKO}.  More precisely, if $f(\tau)$
is a normalized weight $k$ meromorphic modular form on $\SL_2(\Z)$, then the {\it divisor modular form}\footnote{Note that this summation does not include the cusp $i\infty$.} is
\begin{equation}\label{div1}
f^{\dv}(\tau):= \sum_{z\in \SL_2(\Z)\backslash \H} e_{z}\ord_{z}(f) H_{z}(\tau),
\end{equation}
where
$e_{z}:=2/\Stab_z\left(\SL_2(\Z)\right).$
With $\Theta:=\frac{1}{2\pi i}\frac{d}{d\tau}$, Theorem~1 of \cite{BKO} asserts that
\begin{equation}\label{BKOdiv}
f^{\dv}(\tau)=-\frac{\Theta(f\left(\tau)\right)}{f(\tau)}+\frac{k E_2(\tau)}{12}.
\end{equation}

Although these results rely on the fact that $X_0(1)$ has genus 0, there is a natural extension for congruence subgroups. This extension requires polar harmonic Maass forms (see Section~\ref{polar}).  Here we consider the modular curves $X_0(N)$.  For $n\in\N$, we define a Hecke system of  $\Gamma_0(N)$ harmonic Maass functions $j_{N,n}(\tau)$ in Section~\ref{Thm1Proof} which generalize the $j_n(\tau)$. One key property that we note about the functions $j_{N,n}(\tau)$ is their growth in $n$-aspect, which is not easily described in terms of Fourier expansions. Instead, we use ``Ramanujan-like'' expansions, sums of the form
\begin{equation}\label{eqn:Ramanujanlike}
\sum_{\lambda\in \Lambda_{z}}\sum_{(c,d)\in S_{\lambda}} \frac{1}{\lambda^{k}} e\left(-\frac{n}{\lambda}r_{z}(c,d,k)\right) e^{\frac{2\pi n\mathrm{Im}(z)}{\lambda}},
\end{equation}
for some real numbers $r_{z}(c,d,k)$ (see \eqref{eqn:rcd}), $\Lambda_z$ a lattice in $\R$ (see \eqref{eqn:Lambdazdef}), and $S_{\lambda}$ the set of solutions to $Q_{z}(c,d)=\lambda$ for a certain positive-definite binary quadratic form $Q_z$ (see \eqref{eqn:Slambdadef}).

In Section~\ref{polar} we construct weight 2 polar harmonic Maass forms $H_{N,z}^*(\tau)$ which generalize  the $H_{z}(\tau)$.  We have two cases for the $H_{N,z}^*(\tau)$ which we consider separately.
The following theorem summarizes the essential properties of these functions when $z\in \H$.

\begin{theorem}\label{AKNGeneralization}
If $z\in \H$, then $H_{N,z}^*(\tau)$ is a weight 2 polar harmonic Maass form on $\Gamma_0(N)$ which vanishes at all cusps and  has a single simple pole at $z$. Moreover, the following are true:
\begin{enumerate}[leftmargin=*, label={\rm(\arabic*)}]
\item
 If $z\in \H$ and $\im(\tau) > \operatorname{max}\{\mathrm{Im}(z), \frac1{\mathrm{Im}(z)}\}$, then we have that
$$
H_{N,z}^*(\tau) =\frac{3}{\pi\left[\SL_2(\Z):\Gamma_0(N)\right] \im(\tau)}+\sum_{n=1}^{\infty} j_{N,n}(z)e^{2\pi i n \tau}.
$$

\item For $(N,n)=1$, we have
$j_{N,n}(\tau) =j_{N,1}(\tau) \ |\ T(n).$
\item For $n\mid N$, we have
$j_{N,n}(\tau)= j_{\frac{N}{n},1}(n\tau).$

\item  As $n\to \infty$, we have
\[
j_{N,n}\left(\tau\right) = \sum_{\substack{\lambda\in\Lambda_{\tau}\\ \lambda\leq n}}\sum_{(c,d)\in S_{\lambda}} e\left(-\frac{n}{\lambda} r_{\tau}(c,d)\right) e^{\frac{2\pi n \mathrm{Im}(\tau)}{\lambda}} + O_{\tau}(n).
\]
\end{enumerate}
\end{theorem}

\smallskip
\noindent
{\it Five Remarks.}

\noindent
(1) In Theorem \ref{AKNGeneralization} (1), the inequality on $\im(\tau)$ is required for convergence.
\smallskip

\noindent
 (2)   For $N=1$, we have that
$H_{1,z}^*(\tau)= H_z(\tau)-E_{2}^*(\tau)$, where $E_2^*(\tau):=-\frac{3}{\pi \mathrm{Im}(\tau)}+E_{2}(\tau)$ is the usual weight 2 nonholomorphic Eisenstein series, and we have that $j_{1,n}(\tau)=j_n(\tau).$
 \smallskip

\noindent
(3) The sums  (\ref{eqn:Ramanujanlike}) were introduced by Hardy and Ramanujan \cite{HR3} (see also \cite{bby,bia}) in their study of  $1/E_6$.  These results have been generalized \cite{BKFC,BKFC2} to negative weight meromorphic modular forms.  Theorem~\ref{AKNGeneralization} (4) extends these results to weight 0 where the series are not absolutely convergent.

\noindent
(4) Theorem~\ref{AKNGeneralization} (4) gives asymptotics for $j_{N,n}(z)$ in $n$-aspect.
Individual $j_{N,n}(z)$  are easily computed.
By replacing $c>\sqrt{n}/y$ with $c>C$ in \eqref{eqn:expm=0} and \eqref{eqn:expm>0}, one finds that the terms decay like $C^{-1/2+\varepsilon}$ times a power of $n$.  For  $c\leq C$, the expansions in Proposition~\ref{Niebur} decay exponentially in $m$.

\smallskip
\noindent
(5) If $y\geq \im \left(Mz\right)$ for all $M\in\Gamma_0(N)$, then
%for a positive proportion of $n$ (for which the right hand side does not vanish) we have
\begin{equation}\label{JNnAsymp}
j_{N,n}(z)\approx e^{-2\pi inz}+\sum_{\substack{c\geq 1\\ N\mid c}} \sum_{\substack{d\in \Z\\ \gcd(c,d)=1\\ |cz+d|^2=1}} e\left(n\frac{d-a}{c}\right) e^{2\pi in\overline{z}}.
\end{equation}

\smallskip

The second case we consider are those $H_{N,\rho}^*(\tau)$ where $\rho$
is a cusp of $X_0(N)$. These functions are compatible with the $H_{N,z}^*(\tau)$ considered in Theorem~\ref{AKNGeneralization}.  More precisely, since $z\mapsto H_{N,z}^*(\tau)$ is continuous (even harmonic) and $\Gamma_0(N)$-invariant, it follows that
\begin{equation}\label{eqn:Hrhodef}
H_{N,\rho}^*(\tau):=\lim_{z\to \rho} H_{N,z}^*(\tau)
\end{equation}
is well-defined and only depends on the equivalence class of $\rho$.
The next result summarizes these functions' properties. We use the Kloosterman sums $K_{i\infty, \rho}(0,-n;c)$ of \eqref{Krhodef} and the weight $2$ harmonic Eisenstein series $E_{2,N,\rho}^*(\tau)$ with constant term $1$ at $\rho$ and vanishing at all other cusps.

\begin{theorem}\label{thm:HtauCusp}
We have that $H_{N,\rho}^*(\tau)=-E_{2,N,\rho}^*(\tau)$.  Moreover, the following are true:
\begin{enumerate}[leftmargin=*, label={\rm(\arabic*)}]
\item We have
\begin{align*}
	H_{N,\rho}^* (\tau)
	&\hphantom{:}= \frac{3}{\pi\left[\SL_2(\Z):\Gamma_0(N)\right]\im(\tau)}-\delta_{\rho,\infty}+\sum_{n=1}^\infty  j_{N,n}(\rho) e^{2\pi in\tau}, \qquad \text{ with}\\
	j_{N,n}(\rho)
	&:=\lim_{\tau\to \rho}j_{N,n}(\tau)=\frac{4\pi ^2 n}{\ell_\rho}\sum_{\substack{c\geq 1\\ N|c}}\frac{K_{i\infty, \rho}(0,-n;c)}{c^2},
\end{align*}
	where $\delta_{\rho,\infty}:=1$ if $\rho =i\infty$ and $0$ otherwise.
\item  For $\gcd(N,n)=1$, we have
$
j_{N,n}(\rho) =\lim_{\tau\to \rho}j_{N,1}(\tau) \ |\ T(n).
$
\item For $n\mid N$, we have
$
j_{N,n}(\rho)= \lim_{\tau\to \rho} j_{\frac{N}{n},1}(n\tau).
$
\end{enumerate}
\end{theorem}

\noindent
{\it Two Remarks.}

\noindent
(1) Recall that the Fourier expansion in Theorem \ref{AKNGeneralization} (1) is not valid as $z\to i\infty$.
\smallskip

\noindent
(2) The $j_{N,n}(\rho)$ are explicit divisor sums, which we leave to the interested reader to verify. From a generalization of the Weil bound \eqref{eqn:WeilBound} one can obtain $j_{N,n}(\rho)=O(n^{\frac32})$.
\medskip

\begin{comment}
For $N=1$, Theorem~\ref{AKNGeneralization} is related to the reformulation of the denominator formula for the Monster Lie algebra \eqref{ZagierHFunction} via \eqref{eqn:H1relate}.
Note that by the valence formula
\[
\sum_{z\in \SL_2(\Z)\backslash \H\cup \{i\infty\}} e_{z}\ord_{z}(f) = \frac{k}{12},
\]
so that if we replace $H_z(\tau)$ with $H_{1,z}(\tau)$ in \eqref{div1}, then \eqref{eqn:H1relate} implies that \eqref{BKOdiv} simply becomes
\begin{equation}\label{eqn:divN=1}
\sum_{z\in \SL_2(\Z)\backslash \H\cup \{i\infty\}} e_{z}\ord_{z}(f) H_{1,z}=-\frac{\Theta(f)}{f}.
\end{equation}
Moreover, if $X_0(N)$ has genus 0, then the $j_{N,n}(\tau)$ are polynomials in the Hauptmodul $j_{N,1}(\tau)$, and so these identities also may be similarly interpreted in Moonshine.
\end{comment}

We turn to the task of extending (\ref{BKOdiv}) to generic $\Gamma_0(N)$.
Suppose that $f$ is a weight $k$ meromorphic modular form on $\Gamma_0(N)$.
In analogy with (\ref{div1}), we define the {\it divisor polar harmonic Maass form}
\begin{equation}\label{div2}
{f}^{\dv}(\tau):=\sum_{z\in X_0(N)} e_{N,z} \ord_z(f) H_{N,z}^*(\tau),
\end{equation}
where $e_{N,z}:=2/\#\Stab_z(\Gamma_0(N))$ (we let $e_{N,\rho}:=1$).
Generalizing (\ref{BKOdiv}), we show the following.

\begin{theorem}\label{divN}  If $S_2(\Gamma_0(N))$ denotes the space of weight 2 cusp forms on $\Gamma_0(N)$, then
$$
{f}^{\dv}(\tau) \equiv\frac{k}{4\pi \im(\tau)}-\frac{\Theta(f(\tau))}{f(\tau)} \pmod{S_2(\Gamma_0(N))}.
$$
\end{theorem}

\smallskip
\noindent
{\it Three Remarks.}

\noindent
(1) The coefficient of $1/\im(\tau)$ in $H_{N,z}^*(\tau)$ is independent of $z$.  By the valence formula, summing over every element of $X_0(N)$ in the definition of $f^{\dv}(\tau)$ multiplies this constant by $\frac{k}{12}\left[\SL_2(\Z):\Gamma_0(N)\right]$, giving the nonholomorphic correction term on the right-hand side of Theorem \ref{divN}.

\smallskip
\noindent
(2) At first glance, definitions (\ref{div1}) and (\ref{div2}) might appear different for $N=1$.
Indeed, $H_{1,z}^*(\tau)= H_z(\tau)-E_{2}^*(\tau)$, and the sum in $(\ref{div2})$  includes the cusp $i\infty$
whereas $(\ref{div1})$ omits it.
The quasimodular Eisenstein series
$E_2(\tau)$ in (\ref{BKOdiv}) and the valence formula guarantees that they coincide.

\smallskip
\noindent
(3) The formula in Theorem \ref{divN} has already been obtained by Choi using a regularized inner product due to Petersson, but without relating the Fourier coefficients of $f^{\dv}$ to the polar harmonic Maass forms $H_{N,z}^*$ (see Theorem 1.4 of \cite{Choi}).

\medskip

Theorem~\ref{divN} can be used to numerically compute divisors of  meromorphic modular forms $f(\tau)$, which, in general, is a difficult task
 (for example, see \cite{Delaunay}). The idea is simple. The series $-\frac{\Theta\left(f(\tau)\right)}{f(\tau)}$ is the
logarithmic derivative of $f(\tau)$, and this fact converts the points $z\in \H$ in the divisor of $f(\tau)$ into simple poles. These can be identified by the asymptotic properties of the coefficients of $H_{N,z}^*(\tau)$ given in Theorem~\ref{AKNGeneralization}. This follows from Theorem~\ref{divN}
and the fact that coefficients of cusp forms satisfy Deligne's bound.
This method is based on the following immediate corollary to Theorems~\ref{AKNGeneralization}--\ref{divN}.

\begin{corollary}\label{asymptotics}
Suppose that $f(\tau)$ is a meromorphic modular form of weight $k$ on $\Gamma_0(N)$ whose divisor is not supported at cusps. Let $y_1$ be the largest imaginary part of any points in the divisor of $f(\tau)$ lying in $\mathbb H$. Then if $-\frac{\Theta(f(\tau))}{f(\tau)}=:
\sum_{n\gg-\infty}a(n)q^n$ $(q=e^{2\pi i\tau})$, we have that 
$$
y_1=\limsup_{n\rightarrow\infty}\frac{\log |a(n)|}{2\pi n}.
$$
\end{corollary}

\smallskip
\noindent
{\it Two Remarks.}

\noindent
(1) We require $\limsup$ in Corollary~\ref{asymptotics} because the $a(n)$  can vanish on arithmetic progressions.

\smallskip
\noindent
 (2)  It would be interesting to develop a practical algorithm for numerically computing modular form divisors. The idea would be to
carefully peel away poles of $f^{\dv}(\tau)$ in descending order until one is left with a linear combination of functions $E_{N,\rho}^*(\tau).$

\begin{example}
For the Eisenstein series $E_4(\tau)$, we have
\[
-\frac{\Theta(E_4(\tau))}{E_4(\tau)}=-240q + 53280q^2 - 12288960q^3 + 2835808320q^4 - 654403831200q^5+\cdots.
\]
The sequence $\{b(n)\}_{n\geq1}=\{\log|a(n)|/(2\pi n)\}_{n\geq1}$ converges rapidly. Indeed,  $b(2)=0.866066794\dots$, and
$b(10)=0.866025404\dots$ matches the first 16 digits of the limiting value.
%\begin{figure}[H]
%\caption{The first coefficients arising in Corollary~\ref{asymptotics} applied to $E_4$}
%\begin{center}
%\begin{tabular}{c c }
%\hline\hline
%$n$ & $\frac{\log |a_F(n)|}{2\pi n}$\\ [0.7ex]
%\hline
%$1$ & $0.872270776\ldots$
%\\
%$2$ & $0.866066794\ldots$
%\\
%$3$ & $0.866026336\ldots$
%\\
%$4$ & $0.866025423\ldots$
%\\
%$5$ & $0.866025404\ldots$
%\\
%$10$&$0.866025404\dots$
%\\
%\hline
%\end{tabular}
%\end{center}
%\end{figure}
The divisor of $E_4(\tau)$ is supported on a zero at  $\omega:=(-1+\sqrt{-3})/2$.
By \eqref{JNnAsymp}, since $\omega$ lies on the unit circle (implying that the second term on the right-hand side of \eqref{JNnAsymp} appears) for large $n$, $a(n)$ should very nearly be
$\frac13\left(e^{-2\pi in\omega}+2 e^{2\pi in\overline{\omega}}\right)=e^{-2\pi in\omega}$,
which is very easily seen numerically.
\end{example}

\begin{example}
 We consider $f(\tau):=E_4(2\tau)+\frac{\eta^{16}(2\tau)}{\eta^8(\tau)}$, where $\eta(\tau)$ is Dedekind's eta-function. By the valence formula for $\Gamma_0(2)$, it has a single zero, say $z_0$, in $X_0(2)$. We find that
\[
-\frac{\Theta(f(\tau))}{f(\tau)}=
-q-495q^2+659q^3+113233q^4-261211q^5+
\cdots.
\]
After the first $3000$ terms the sequence $\log |a(n)|/(2\pi n)$ stabilizes
and offers $\im(z_0)\approx 0.4357$. As $f(\tau)$ has real coefficients and there is only one zero, $-\overline{z}_0$ must be $\Gamma_0(2)$-equivalent to $z_0$.
%Choosing the fundamental domain
%$$
%\left\lbrace \text{$z\in\H $:  $-\frac12\leq \re z\leq \frac12$ and $\forall M\in\Gamma_0(2)$: $%\Bigl(\im\left(Mz\right)\geq \im z$ and $\im\left(Mz\right)> \im z$ if $\re z <0\Bigr)$}\right\rbrace,
%$$
Thus, either $\mathrm{Re}(z)=0$, $1/2$ or $z$ lies on the arc $|2z-1|=1$. The first two cases are easily excluded by the sign patterns of $a(n)$, and the zero on the arc is easily approximated as $z_0\approx 0.2547 +0.4357i$.
\end{example}

This paper is organized as follows. In Section~\ref{polar} we construct the
weight $2$ polar harmonic Maass forms $H_{N,z}^*(\tau)$. In Section~\ref{Thm1Proof} we relate its Fourier
coefficients to the values of the weight 0 weak Maass forms at $\tau=z$.  In
other words, we prove Theorems~\ref{AKNGeneralization} and \ref{thm:HtauCusp}.  In
Section~\ref{ProofdivN} we show Theorem~\ref{divN}.

\section{Weight 2 Polar Harmonic Maass forms}\label{polar}

\subsection{ The \texorpdfstring{$H_{N,z}^*(\tau)$}{HNz*(tau)} when \texorpdfstring{$z\in \H$}{z in H}}

Define for $z, \tau\in \H$ and $s\in\C$ with $\re(s)>0$
\begin{equation}\label{Pns}
	P_{N, s}(\tau, z):=\sum_{M\in \Gamma_0(N)} \frac{\varphi_s\left(M\tau, z\right)}{j\left(M, \tau\right)^2|j\left(M, \tau\right)|^{2s}}
\end{equation}
with $j(\begin{psmallmatrix}
a&b\\
c&d
\end{psmallmatrix},\tau):=(c\tau+d)$ and 
\begin{equation*}
	\varphi_s(\tau, z):=\left(\mathrm{Im}(z)\right)^{1+s}(\tau-z)^{-1}(\tau-\overline{z})^{-1}\left|\tau-\overline{z}\right|^{-2s}.
\end{equation*}
These functions were introduced and investigated in the $z$-variable  in \cite{BK}, where it was shown that these are {\it polar harmonic Maass forms}.  These functions are allowed to have poles in the upper half plane instead of only at the cusps.  In this paper, we are interested in properties of $P_{N,s}(\tau, z)$ as functions of $\tau$.  A direct calculation shows that for $L\in\Gamma_0(N)$
$$
P_{N,s}\left(L\tau, z\right)=j\left(L, \tau\right)^2|j\left(L, \tau\right)|^{2s}P_{N, s}(\tau, z).
$$

In \cite{BK} it was shown, by a lengthy calculation, that the function $P_{N,s}(\tau, z)$ has an analytic continuation to $s=0$, which we denote by $\im (z)\Psi_{2,N}(\tau, z)$.
Let $\mathcal H_k(\Gamma_0(N))$ be the space of polar harmonic Maass forms with respect to $\Gamma_0(N)$.  Lemma 4.4 of \cite{BK} then states that $z\mapsto \im (z)\Psi_{2,N}(\tau,z)\in \mathcal{H}_0(\Gamma_0(N))$.  In the $\tau$ variable, these functions are also polar harmonic Maass forms, as the next proposition shows. For this, for $w\in\C$, let $e(w):=e^{2\pi i w}$, and
\begin{equation*}
	K(m,n;c):=\sum_{\substack{a, d\pmod c \\ ad\equiv 1\pmod c }} e\left(\frac{m d +na}{ c}\right).
\end{equation*}
Moreover, $I_k$ and $J_k$ denote the usual $I$- and $J$-Bessel functions. The following proposition can be obtained by a careful inspection of the proof of Theorem 3.1 of \cite{BK}.
\begin{proposition}\label{prop:Psiexp}
We have that $\tau\mapsto \mathrm{Im}(z)\Psi_{2,N}(\tau,z)\in\mathcal H_2(\Gamma_0(N))$. For $\im(\tau) > \operatorname{max}\{\mathrm{Im}(z), \frac{1}{\mathrm{Im}(z)}\}$, its Fourier expansion (in $\tau$) has the form 
\begin{align*}
	&y\Psi_{2,N}(\tau,z)
	=-\frac{6}{[\SL_2(\Z):\Gamma_0(N)]\im(\tau)}-2\pi\sum_{m\geq 1}\left(e^{-2\pi imz}-e^{-2\pi im\overline{z}}\right)e^{2\pi im\tau}\\
	&\ \ \quad-4\pi ^2 \sum_{m\geq 1}\sum_{\substack{n,c\geq 1 \\ N|c}}\sqrt{\frac{m}{n}}\frac{K(m,-n;c)}{c}I_1\left(\frac{4\pi\sqrt{mn}}{c}\right)e^{2\pi inz}e^{2\pi im\tau}\\
	&\ \ \quad-4\pi ^2 \sum_{m\geq 1}\sum_{\substack{n,c\geq 1\\ N|c}}\sqrt{\frac{m}{n}}\frac{K(m,n;c)}{c}J_1\left(\frac{4\pi\sqrt{mn}}{c}\right)e^{-2\pi in\overline{z}}e^{2\pi im\tau}
	-8\pi ^3\sum_{m\geq 1}m\sum_{\substack{c\geq 1\\ N|c}}\frac{K(m,0;c)}{c^2}e^{2\pi im\tau}.
\end{align*}
\end{proposition}

We then set
\begin{equation}\label{eqn:Hdef}
H_{N,z}^*(\tau):=-\frac{\im (z)}{2\pi}\Psi_{2,N}(\tau,z).
\end{equation}

\begin{remark}
We have, as $\tau\to z$,
\begin{equation}\label{PP}
H^*_{N,z}(\tau)=\frac{1}{2\pi i e_{N, z}}\frac{1}{\tau-z}+O(1).
\end{equation}
\end{remark}

\subsection{The \texorpdfstring{$H_{N,z}^*(\tau)$}{HNz*(tau)} for cusps} \label{sec:harmonic}

  We require the Fourier expansion of the functions $H_{N,\rho}^*(\tau)$ defined in \eqref{eqn:Hrhodef}.  For any cusp $\rho$ of $\Gamma_0(N)$, denote by $\ell_\rho$ the cusp width and let $M_\rho$ be a matrix in $\SL_2(\Z)$ with $\rho=M_\rho i \infty$. For two cusps $\mathfrak{a},\mathfrak{b}$ of $\Gamma_0(N)$, the generalized Kloosterman sums are
\begin{equation}\label{Krhodef}
	K_{\mathfrak{a}, \mathfrak{b}}(m,n;c):=\sum_{\left.\left.\left(\begin{smallmatrix} a & b \\ c & d\end{smallmatrix}\right) \in\Gamma_\infty^{\ell_\mathfrak{a}}\right\backslash M_\mathfrak{a}^{-1}\Gamma_0(N)M_\mathfrak{b}\right/\Gamma_\infty^{\ell_\mathfrak{b}}} e\left(\frac{m d}{\ell_\mathfrak{b} c} +\frac{na}{\ell_\mathfrak{a} c}\right)
\end{equation}
with $\Gamma_\infty:=\left\{\pm\begin{psmallmatrix}
	1 & n\\
	0 & 1
\end{psmallmatrix}:n\in\Z\right\}$. 
Note that we have $K_{i\infty,i\infty}(m,n;c)=K(m,n;c)$.

\begin{lemma}[Lemma 5.4 of \cite{BK}]\label{lem:HNrhoexp}
We have
\[
H_{N,\rho}^*(\tau) = \frac{3}{\pi \left[\SL_2(\Z):\Gamma_0(N)\right] \im(\tau)}  -\delta_{\rho,\infty} +\frac{4\pi ^2}{\ell_{\rho}}\sum_{n\geq 1} n \sum_{c\geq 1} \frac{K_{\rho, i\infty}(n,0;c)}{c^2} e^{2\pi i n \tau}.
\]
\end{lemma}
%\begin{proof}[Sketch of proof]
%One may compute the Fourier expansions of $z\mapsto y\Psi_{2,N}(\tau,z)$ at the cusp $\rho$, then taking the limit to obtain the Fourier expansion of $H_{N,\rho}^*$.  The expansions as a function of $z$ and the corresponding limits were evaluated in Lemma 5.4 of \cite{BK}.
%\end{proof}

The Fourier expansions in Lemma \ref{lem:HNrhoexp} yield a connection with the harmonic weight $2$ Eisenstein series $E_{2,N,\rho}^*(\tau)$ for $\Gamma_0(N)$.  For $\re (s) >0$, define
\begin{equation}\label{Hecke}
E_{2, N,\rho,s}^*(\tau):=\sum_{M\in \Gamma_\rho\setminus \Gamma_0(N)} j\left(M_\rho M, \tau\right)^{-2}\left|j\left(M_\rho M, \tau\right)\right|^{-2s}.
\end{equation}
Using the Hecke trick, it is well-known (cf. Satz 6 of \cite{Hecke}) that $E_{2,N,\rho,s}^*(\tau)$ has an analytic continuation to $s=0$, denoted by $E_{2,N,\rho}^*(\tau)$. From the Fourier expansion given in Theorem 1 of \cite{Smart}, we see
\begin{equation}\label{HrhoE2rho}
H_{N,\rho}^*(\tau)=-E_{2,N,\rho}^*(\tau).
\end{equation}

%These Eisenstein series have the following Fourier expansion.

%\begin{proposition}[Theorem 1 of \cite{Smart}]\label{prop:HrhoE2rho}
%We have that $H_{N,\rho}^*=-E_{2,N,\rho}^*$, and
%\[
%E_{2,N,\rho}^*(\tau) = -\frac{3}{\pi \left[\SL_2(\Z):\Gamma_0(N)\right] \im(\tau)}  +\delta_{\rho,\infty} -\frac{4\pi ^2}{\ell_{\rho}}\sum_{n\geq 1} n \sum_{c\geq 1} \frac{K_{\rho, i\infty}(n,0;c)}{c^2} e^{2\pi i n \tau}.
%\]
%\end{proposition}

\section{The \texorpdfstring{$j_{N,n}(z)$}{jNn(z)} and the proofs of Theorems \ref{AKNGeneralization} and \ref{thm:HtauCusp}}\label{Thm1Proof}
\subsection{The \texorpdfstring{$j_{N,n}(z)$}{jNn(z)}}

The functions $j_{N,n}(z)$ are constructed as analytic continuations of Niebur's Poincar\'e series \cite{Niebur}. To be more precise, set for $n\in \N$ and $\re(s)>1$
\begin{align*}
	F_{N, -n, s}(z)
	:=\sum_{M\in\Gamma_\infty \backslash\Gamma_0(N)}e\left(-n\re (Mz)\right)\im(Mz)^\frac12 I_{s-\frac12}\left(2\pi n \im(Mz)\right).
\end{align*}
These functions are {\it weak Maass forms} of weight $0$; instead of being annihilated by $\Delta_0$, they have eigenvalue $s(1-s)$.
 To obtain an analytic continuation to $s=1$, one computes the Fourier expansion of $F_{N,-n,s}(z)$ and uses
$$
\lim_{s\to 1} y^{\frac12}I_{s-\frac12}\left(2\pi ny\right)=y^{\frac12}I_{\frac12}\left(2\pi ny\right)=\frac{1}{\pi\sqrt{n}}\sinh\left(2\pi ny\right)=\frac{e^{2\pi ny}-e^{-2\pi ny}}{2\pi\sqrt{n}}.
$$
\rm
\begin{proposition}[Theorem 1 of \cite{Niebur}]\label{Niebur}
	The function $F_{N, -n, s}(z)$ has an analytic continuation $F_{N, -n}(z)$ to $s=1$ and $F_{N, -n}(z)\in \mathcal H_0 (\Gamma_0(N))$. We have the Fourier expansion
	\begin{align*}
		F_{N, -n}(z)=\frac{e^{-2\pi inz}-e^{-2\pi in\overline{z}}}{2\pi\sqrt{n}}+c_N(n,0)+\sum_{m\ge1}\left(c_N(n,m)e^{2\pi imz}+c_N(n,-m)e^{-2\pi im\overline{z}}\right),
	\end{align*}
	where the coefficients are given by
$$
c_N(n,m):= \sum_{\substack{c\geq 1 \\ N|c}}\frac{K(m,-n;c)}{c}\times\begin{cases}\frac{1}{\sqrt{m}} I_1\left(\frac{4\pi\sqrt{mn}}{c}\right) & \text{if $m> 0$,}\\
\frac{2\pi \sqrt{n}}{c}& \text{if $m= 0$,}\\
\frac{1}{\sqrt{|m|}} J_1\left(\frac{4\pi\sqrt{|m|n}}{c}\right) & \text{if $m< 0$.}
\end{cases}
$$
\end{proposition}
	We then define the functions $j_{N,n}(z)$ by
	\begin{align}
j_{N,n}(z):=2\pi\sqrt{n}F_{N, -n}(z).\label{JNnDefinition}
	\end{align}
For $N=1$, we recover the $j_n(z)$ from the introduction up to the constant $2\pi \sqrt{n} c_1(n,0)=24 \sigma_1(n)$.

\subsection{Proofs of Theorems~\ref{AKNGeneralization} and \ref{thm:HtauCusp}}

In order to formally state Theorem \ref{AKNGeneralization} (4), for an arbitrary solution $a,b\in\Z$ to $ad-bc=1$, we define
\begin{align}\label{eqn:rcd}
r_z(c,d)&:=ac|z|^2+(ad+bc)\re(z)+bd,\\
\label{eqn:Lambdazdef}\Lambda_z&:=\left\{\alpha^2 |z|^2 + \beta \re(z) + \gamma^2>0: \alpha,\beta,\gamma\in \Z\right\},\\
\label{eqn:Slambdadef}S_{\lambda}&:=\left\{ (c,d)\in N\N_0\times \Z:\ \gcd(c,d)=1\text{ and }Q_{z}(c,d)=\lambda\right\},\\
\nonumber Q_{z}(c,d)&:=c^2|z|^2+2cd\re(z)+d^2.
\end{align}
Note that although $r_{z}(c,d)$ is not uniquely determined, $e(-nr_{z}(c,d)/Q_z(c,d))$ is well-defined.
\begin{proof}[Proof of Theorem \ref{AKNGeneralization}]
(1) For $m\in\N$, inspecting the expansions in Propositions \ref{prop:Psiexp} and \ref{Niebur} yields that $2\pi \sqrt{m}F_{N, -m}(z)$ is the coefficient of $e^{2\pi im\tau}$ in $-\im(z)\Psi_{2,N}(\tau,z)/(2\pi)$, yielding the claim.
\noindent

\noindent
(2)  Since $\gcd(N,n)=1$,  $T(n)$ commutes with the action of $\Gamma_0(N)$, and so it suffices to show that (by analytic continuation)
$f_n(z)=f_1(z)\mid T(n),$ where
$$
f_n(z)=f_{n,s}(z):= e\left(-n\re (z)\right)\left(n\im(z)\right)^{\frac12} I_{s-\frac12} \left(2\pi n\im(z)\right).
$$
Since $f_n(z)$ is a (nonholomorphic) Fourier coefficient, we use the formula ($\gcd(n,N)=1$)
\begin{equation}\label{eqn:Hecke}
f(z)\mid T(n)=n\sum_{m\in\Z}\sum_{d\mid \gcd(m,n)}\frac{a_{\frac{d^2}{n}\im(z)}\left(\frac{mn}{d^2}\right)}{d} e^{2\pi imz},
\end{equation}
if $f(z)=\sum_{m\in\Z}a_{\im(z)}(m)e^{2\pi imz}$ is a nonholomorphic modular function of weight $0$.
Write
	$f_n(z)=f_n^*(\im(z))e^{-2\pi inz}$ with
	$f_n^*(y):=(ny)^{\frac12}I_{s-\frac12}(2\pi ny)e^{-2\pi ny}.$
The $m$th  coefficient in \eqref{eqn:Hecke} vanishes unless $m=-n$. Moreover, only $d=n$ contributes, giving
\[
f_1(z)\mid T(n)=f_n^*\left(n\im(z)\right)e^{-2\pi inz}=\left(n\im(z)\right)^{\frac12}I_{s-\frac12}\left(2\pi n\im(z)\right)e^{-2\pi i n\re (z)}=f_n(z).
\]

\noindent
(3)  For $n\mid N$, we rewrite
\[
\sum_{M\in\Gamma_\infty\setminus\Gamma_0(N)}f_n(Mz)=\sum_{M\in\Gamma_\infty\setminus\Gamma_0(N)} f_1(nMz).
\]
Now, with $M=\begin{psmallmatrix}a&b\\c&d \end{psmallmatrix}\in\Gamma_0(N)$, we have
$
nMz=\frac{anz+bn}{\frac{c}{n}nz+d}
$
and $\begin{psmallmatrix}a&bn\\ \frac{c}{n}&d \end{psmallmatrix}$ runs through $\Gamma_\infty\setminus\Gamma_0(\frac{N}{n})$ if $M$ runs through $\Gamma_\infty\setminus\Gamma_0(N)$, implying the claim for $n\mid N$.

\noindent
(4)  We first rewrite the claimed asymptotic formula in terms of the corresponding points $Mz$ with $M=\begin{psmallmatrix}a&b\\c&d\end{psmallmatrix} \in\Gamma_{\infty}\backslash \Gamma_0(N)$.   Directly plugging in and simplifying yields $r_z(c,d)/Q_z(c,d) = \re(Mz)$ and $\im(z)/Q_z(c,d)=\im(Mz)$, so the claim in Theorem \ref{AKNGeneralization} (4) is equivalent to
\begin{equation}\label{eqn:jasymp}
j_{N,n}(z)= \sum_{\substack{M\in \Gamma_{\infty}\backslash \Gamma_0(N)\\ n\im(Mz)\geq \im(z)}} e^{-2\pi in Mz}+O_{z}(n).
\end{equation}

\noindent
In order to show \eqref{eqn:jasymp}, we only expand the Fourier expansion for large $c$.  That is to say, we write
\begin{equation}\begin{split}\label{eqn:jexp}
j_{N,n}(z)&=2 \sum_{\substack{1\leq c\leq \frac{\sqrt{n}}{\im(z)}\\ N\mid c}} \sum_{\substack{d\in\Z\\ \gcd(c,d)=1}} e(-n\re(Mz))\sinh(2\pi n\im(Mz))\\
&\quad+2\pi \sqrt{n}\sum_{\substack{c>\frac{\sqrt{n}}{\im(z)}\\ N|c}}\sum_{m\geq 1}\frac{K(m,-n;c)}{\sqrt{m}c}I_1\left(\frac{4\pi\sqrt{mn}}{c}\right)e^{2\pi imz} +4\pi^2n\sum_{\substack{c>\frac{\sqrt{n}}{\im(z)}\\ N|c}}\frac{K(0,-n;c)}{c^2 }\\
&\quad+2\pi \sqrt{n}\sum_{\substack{c>\frac{\sqrt{n}}{\im(z)}\\ N|c}}\sum_{m\geq 1}\frac{K(-m,-n;c)}{\sqrt{m}c}J_1\left(\frac{4\pi\sqrt{mn}}{c}\right)e^{-2\pi im\overline{z}}.
\end{split}\end{equation}

In order to obtain \eqref{eqn:jasymp}, we split the main terms with $n\im(Mz)\geq \im(z)$ off and rewrite
\begin{equation}\label{eqn:mainasym}
2\sinh(2\pi n\im(Mz))= e^{2\pi n\im(Mz)}-e^{-2\pi n\im(Mz)}.
\end{equation}
The second term above is obviously bounded.
% by $e^{-2\pi y}$, using that $n\im(Mz)\geq y$.
Since
$$
\im(z)\leq n\im(Mz)=\frac{n\im(z)}{c^2\im(z)^2+(d+c\re (z))^2}
$$
implies that $c\leq \sqrt{n}/\im(z) \ll_{z} \sqrt{n}$ and $|d|\leq |c\re (z)|+\sqrt{n\im(z)}\ll_z \sqrt{n}$, the contribution to the error from the sum of the second terms in \eqref{eqn:mainasym} yields an error of at most $O_z(n)$.

For the second, third, and fourth sums in \eqref{eqn:jexp}, we use the Weil bound for Kloosterman sums
\begin{equation}\label{eqn:WeilBound}
|K(m,-n;c)|\leq \sqrt{\gcd(m,n,c)}\sigma_0(c)\sqrt{c}\ll \begin{cases} \sqrt{n}c^{\frac12+\varepsilon}&\text{if }m=0,\\
\sqrt{|m|} c^{\frac{1}{2}+\varepsilon} &\text{if }m\neq 0.
\end{cases}
\end{equation}
For the third sum in \eqref{eqn:jexp}, this gives
\begin{equation}\label{eqn:expm=0}
2\pi\sqrt{n}\sum_{\substack{c>\frac{\sqrt{n}}{\im(z)}\\ N|c}}\frac{K(0,-n;c)}{c^2 }\ll n \sum_{\substack{c>\frac{\sqrt{n}}{\im(z)}\\ N|c}} c^{-\frac{3}{2}+\varepsilon}\ll_{z} n^{\frac{3}{4}+\varepsilon}.
\end{equation}
Next note that for $x\geq 0$ we have $|J_1(x)|\leq I_1(x)$ by their series expansions. Since $x\mapsto \frac{I_1(x)}{x}$
is monotonically increasing and grows at most exponentially, the contribution from the second and fourth terms in \eqref{eqn:jexp} may be bounded by, using \eqref{eqn:WeilBound},
\begin{equation}\label{eqn:expm>0}
\begin{split}
&\ll \sum_{\substack{c>\frac{\sqrt{n}}{\im(z)}\\ N|c}}\sum_{m\geq 1}\frac{|K(\pm m,-n;c)|}{\sqrt{m}c}I_1\left(\frac{4\pi\sqrt{mn}}{c}\right)e^{-2\pi m\im(z)}\\
&\ll \sqrt{n}\sum_{\substack{c>\frac{\sqrt{n}}{\im(z)}\\ N|c}}\sum_{m\geq 1}\frac{|K(\pm m,-n;c)|}{c^2}\frac{I_1\left(4\pi \im(z)\sqrt{m}\right)}{4\pi \im(z)\sqrt{m}}e^{-2\pi m\im(z)} \\
&\ll \sqrt{n} \sum_{m\geq 1}  I_1\left(4\pi \im(z)\sqrt{m}\right) e^{-2\pi m\im(z)}\ll \sqrt{n}.
\end{split}\end{equation}

It remains to bound the terms in the first sum in \eqref{eqn:jexp} with $|cz+d|^2>n$.  Since each term gives a constant contribution, the terms with $|d|< \sqrt{n}+|c\re (z)|$ give an error term of at most $O_z(n)$.

We finally assume that $|d|\geq \sqrt{n}+|c\re(z)|$.  Since
$x\mapsto\frac{\sinh(x)}{x}$ is monotonically increasing and $|cz+d|^2>n$, the remaining terms contribute 
\begin{multline*}
\left|\sum_{\substack{c\leq \frac{\sqrt{n}}{\im(z)} \\ N|c}}\sum_{\substack{|d|\geq \sqrt{n}+|c\re(z)|\\ \gcd(c,d)=1, }}e\left(-n\re \left(Mz\right)\right)\sinh\left(2\pi n\im \left(Mz\right)\right)\right|
\leq\sum_{c\leq \frac{\sqrt{n}}{\im(z)}}\sum_{|d|\geq \sqrt{n}+|cx|
}\sinh\left(\frac{2\pi n\im(z)}{|cz+d|^2}\right) \\
\leq\sum_{c\leq \frac{\sqrt{n}}{\im(z)}}\sum_{|d|\geq \sqrt{n}}\sinh\left(\frac{2\pi n\im(z)}{d^2}\right)
\leq 2\pi\sqrt{n}  \sum_{d\geq\sqrt{n}}\frac{ n}{d^2} \frac{\sinh\left(2\pi \im(z)\right)}{2\pi \im(z)} =O_z(n),
\end{multline*}
This implies that the terms in the first sum in \eqref{eqn:jexp} with $|cz+d|^2>n$ contribute $O_z(n)$.
\end{proof}

\begin{proof}[Proof of Theorem \ref{thm:HtauCusp}]
(1)
Let $K_s$ denote the usual $K$-Bessel function.
Expanding $F_{N, -n,s}(z)$ at the cusp $\rho$ as in Section 3.4 of \cite{Iwaniec}, we obtain
\begin{align*}
	F_{N, -n,s}\left(M_\rho z\right)
	&\hphantom{:}=
 \frac{c_{\rho,s}(n,0)}{2s-1}\left(\im(z)\right)^{1-s}
\!+\!\sum_{m\in\Z\setminus\{0\}}c_{\rho,s}(n,m)e^{2\pi im\frac{\re(z)}{\ell_\rho}}\!\left(\im(z)\right)^{\frac{1}{2}}K_{s-\frac{1}{2}}\!\left(\frac{2\pi|m|\im(z)}{\ell_{\rho}}\right),
 \\\text{ with  }
	c_{\rho,s}(n,m)
	&:= \sum_{c\geq 1 }K_{i\infty,\rho}(m,-n;c)
\times\begin{cases}\frac{2}{c\sqrt{\ell_{\rho}}} I_{2s-1}\left(\frac{4\pi\sqrt{mn}}{\ell_\rho c}\right) & \text{if $m> 0$,}\\
	\frac{2\pi^{s} n^{s-\frac{1}{2}}}{\ell_\rho^s c^{2s}\Gamma(s)}
& \text{if $m= 0$,}\\
	\frac{2}{c\sqrt{\ell_{\rho}}} J_{2s-1}\left(\frac{4\pi\sqrt{|m|n}}{\ell_\rho c}\right) & \text{if $m< 0$,}
	\end{cases}
\end{align*}
The right-hand side is analytic at $s=1$, which gives the expansion of $ F_{N, -n}(z)$ at $\rho$.
Plugging in $K_{\frac12}(y)= \sqrt{\frac{\pi}{2y}}e^{-y}$ and taking the limit $z\to i\infty$, we obtain
\begin{equation}\label{eqn:jNnrho}
j_{N,n}(\rho)=2\pi\sqrt{n}\lim_{s\rightarrow 1^+}c_{\rho,s}(n,0)=\frac{4\pi ^2 n}{\ell_\rho}\sum_{c\geq 1}\frac{K_{i\infty,\rho}(0,-n;c)}{c^2}.
\end{equation}
We have $K_{i\infty,\rho}(0,-n;c)=K_{\rho, i\infty}(n,0;c)$, since $M=\left(\begin{smallmatrix}a&b\\c&d\end{smallmatrix}\right)$ runs through $\Gamma_0(N)M_\rho / \Gamma_\infty ^{\ell_\rho}$ iff $-M^{-1}=\left(\begin{smallmatrix}-d&b\\c&-a\end{smallmatrix}\right)$ runs through $\Gamma_\infty ^{\ell_\rho}\backslash M_\rho ^{-1}\Gamma_0(N)$ in \eqref{Krhodef}. Hence \eqref{HrhoE2rho} 
%Proposition \ref{prop:HrhoE2rho} 
yields the claim.
\noindent

\noindent
(2) and (3)  follow by taking limits $\tau\rightarrow\rho$ in Theorem \ref{AKNGeneralization} (2) and (3), respectively.  Using the growth in $n$ of $j_{N,n}(\rho)$ from \eqref{eqn:jNnrho}, these limits may be taken termwise.

\end{proof}

\section{Proof of Theorem~\ref{divN}}\label{ProofdivN}
\begin{proof}[Proof of Theorem \ref{divN}]
%The idea is to show that
%\begin{equation}\label{claim}
%	-\frac{\Theta(f)}{f}\equiv \sum_{x\in X_0(N)}e_{N,z}\ord_z (f)H_{N,z}\pmod {S_2\left(\Gamma_0(N)\right)}.
%\end{equation}
%The claim then follows by taking the modular completion on both sides.

We show that the difference of both sides has no poles in $\H$ and decays towards the cusps. We start by considering the points in $\H$. One easily computes that the residue of $-\frac{\Theta\left(f(\tau)\right)}{f(\tau)}$ at $\tau=z$ equals $\frac{1}{2\pi i}\ord_z(f)$. Using \eqref{PP} gives that the principal part 
%in \eqref{DivThetaEq} 
at $z$ agrees. In a cusp $\rho$ one similarly sees that $\frac{\Theta\left(f(\tau)\right)}{f(\tau)}$ has no pole and its constant term equals $\ord_\rho(f)$. Using that the constant term of $H_{N,z}^*(\tau)$ at $\rho$ is $-1$ then gives the claim.
\end{proof}

\end{document}